\documentclass[a4paper,11pt]{article}
\usepackage[english]{babel}
\usepackage{paralist,url,verbatim,anysize,enumitem}
\usepackage{amscd}
\usepackage{xcolor} 
\usepackage{subfig} 
\usepackage{svg} 
\usepackage{mathrsfs, mathtools}
\usepackage{cancel} 
\usepackage{xfrac,faktor}
\usepackage[e]{esvect}
\usepackage{amsmath,amsthm,amssymb,amsfonts}
\usepackage{graphicx, graphics}
\usepackage[bookmarksopen=true]{hyperref}
\hypersetup{colorlinks=true, citecolor=blue}
\usepackage{bbm}
\usepackage[font=small,labelfont=bf]{caption}
\usepackage{bm} 
\usepackage[symbol*]{footmisc} 
\usepackage{tikz} 
\usepackage[]{geometry}
\numberwithin{equation}{section} 
\numberwithin{figure}{section}

\theoremstyle{plain}
\newtheorem{theorem}{\bf Theorem}[]

\newtheorem{corollary}[theorem]{Corollary}
\newtheorem{lemma}[theorem]{Lemma}
\newtheorem{proposition}[theorem]{Proposition}

\newtheorem{thmnonumber}{\bf Theorem}

\theoremstyle{definition}
\newtheorem{example}[theorem]{Example}

\newtheorem{remark}[theorem]{Remark}

\newtheorem{definition}[theorem]{Definition}

\newcommand\eqdef{\mathrel{\overset{\makebox[0pt]{\mbox{\normalfont\tiny\sffamily def}}}{=}}}

\definecolor{mypink}{RGB}{215, 5, 234}
\definecolor{lemonchiffon}{RGB}{255, 250, 205}

\renewcommand{\textdagger}{$**$}
\renewcommand{\textdaggerdbl}{$**$}
\renewcommand{\textparagraph}{$*$}
\DefineFNsymbolsTM{otherfnsymbols}{%
  \textparagraph \mathparagraph
  \textdagger    \dagger
  \textdaggerdbl \ddagger
  \textasteriskcentered *
  \textbardbl    \|%
}%
\setfnsymbol{otherfnsymbols}

\begin{document}

\author{Bruno Benedetti \thanks{Supported by NSF Grant  1855165, ``Geometric Combinatorics and Discrete Morse Theory''.}\\
\small Dept. of Mathematics\\
\small University of Miami\\
\small bruno@math.miami.edu
\and Marta Pavelka\\
\small Dept. of Mathematics\\
\small University of Miami\\
\small pavelka@math.miami.edu}

\title{2-LC triangulated manifolds are exponentially many}
\maketitle

\begin{abstract}
We introduce ``$t$-LC triangulated manifolds'' as those triangulations obtainable from a tree of $d$-simplices by recursively identifying two boundary $(d-1)$-faces whose intersection has dimension at least $d-t-1$. 
The $t$-LC notion interpolates between the class of LC manifolds introduced by Durhuus-Jonsson (corresponding to the case $t=1$), and the class of all manifolds (case $t=d$).
Benedetti--Ziegler proved that there are at most  $2^{d^2 \, N}$ triangulated $1$-LC  $d$-manifolds with $N$ facets.  
Here we prove that there are at most  $2^{\frac{d^3}{2}N}$ triangulated $2$-LC  $d$-manifolds with $N$ facets. This extends an intuition by Mogami for $d=3$ to all dimensions.

We also introduce ``$t$-constructible complexes'', interpolating between constructible complexes (the case $t=1$) and all complexes (case $t=d$). We show that all $t$-constructible pseudomanifolds are $t$-LC, and that all $t$-constructible complexes have (homotopical) depth larger than $d-t$. This extends the famous result by Hochster that constructible complexes are (homotopy) Cohen--Macaulay.
\end{abstract}

\section{Introduction}
Since the Sixties Tullio Regge \cite{PonzanoRegge, Regge, Regge1} and many other physicists and mathematicians, cf.~e.g.~\cite{ADJ}, \cite{AJL}, \cite{Loll}, have worked to develop a discrete version of quantum gravity. In Weingarten's dynamical triangulations (or ``DT'') setup \cite{Weingarten}, smooth manifolds are approximated by equilateral triangulations. This allows to translate all metric aspects, such as curvature and volume, into simpler combinatorial calculations; for example, the partition function for gravity, which is a path integral over all possible metrics, becomes an infinite sum over all triangulations. The downside of this powerful simplification method is a convergence issue. For example, the partition function diverges to infinity, unless one restricts the sum to  triangulations into a certain class, and such class happens to have exponential size. In fact, for any fixed $d \ge 2$,  there are  more than exponentially many triangulated $d$-manifolds with $N$ facets. Here two triangulations are considered equal if they are `combinatorially isomorphic': That is, if up to relabeling the vertices they have the same face poset.

In an important step for this program, Durhuus and Jonsson  \cite{Durhuus} defined ``locally constructible'' (LC) manifolds as those triangulated manifolds obtainable from a tree of $d$-simplices by recursively identifying two boundary facets whose intersection has {\em codimension one}. They proved that LC 3-manifolds are exponentially many \cite{Durhuus}; and also in higher dimensions, LC $d$-manifolds are less than $2^{d^2 \, N}$ \cite{BBLc}. Since all polytope boundaries are LC, this idea lead to a first proof that polytopes with $N$ facets, in fixed dimension, are exponentially many \cite{BBLc}. 

Here we define \emph{$2$-LC manifolds} as those obtainable from a tree of $d$-simplices by recursively identifying two boundary facets that intersect in \emph{codimension one or two}. We prove that this broader class has also exponential size: 

    \begin{thmnonumber}[Theorem \ref{thm:main}]
    \label{theorem:MainTheoremI}
        For fixed $d \geq 3$, the number of combinatorially distinct 
        $2$-LC $d$-manifolds with $N$ facets is smaller than $2^{\frac{d^3}{2}N}$.
    \end{thmnonumber}
    
The bound can be extended also to 2-LC \emph{quasi-manifolds}, which are pseudomanifolds with particularly nice face links (Theorem \ref{thm:mainQuasi}), but not to arbitrary 2-LC pseudomanifolds, which are more than exponentially many (Remark \ref{rem:pseudo}).  Theorem \ref{theorem:MainTheoremI} gives a precise mathematical formulation and extends to all dimensions an intuition by Mogami \cite{Mogami}, who worked on 2-LC $3$-spheres. The crucial ingredient for this novel exponential upper bound is the planarity of the links of all $(d-3)$-faces. In general, the link of a $k$-dimensional face in a triangulated $d$-dimensional manifold (without boundary) is a homology-sphere of dimension $(d-k-1)$. However, since homology-spheres that are not spheres exist only in dimension $3$ and higher, when $k=d-3$ all links of  $(d-3)$-faces are indeed homeomorphic to $S^2$.
    
This brings topology into the picture.  Durhuus and Jonsson conjectured in 1995 that all $3$-spheres and $3$-balls are LC \cite{Durhuus}. The conjecture was disproved in 2011 by the first author and Ziegler \cite{BBLc}. The weaker conjecture by Mogami \cite{Mogami} that all $3$-balls are 2-LC was also recently disproved by the first author \cite{BBMogami}. Thus there is little hope that these combinatorial cutoffs may encompass entire topologies. But there are  two other reasons why the LC notion is of mathematical importance,  beyond the enumerative aspect mentioned above:
 \begin{compactenum}[ (a) ]
\item All LC-triangulable manifolds are simply-connected, and conversely, all simply-connected PL manifolds of  dimension $\ne 4$ admit an LC subdivision \cite{BBSmooth};
\item All shellable and all constructible manifolds are LC \cite{BBLc}.
\end{compactenum} 

Both results above are still valid if one replaces ``LC'' with ``2-LC''. This triggers a natural curiosity, namely, whether for the result (b) above, for the 2-LC case, one could say more. Perhaps the `constructible' assumption can be weakened?

To answer this curiosity, we define more generally  ``$t$-LC triangulated manifolds'' as those obtainable from a tree of $d$-simplices by recursively identifying two boundary $(d-1)$-faces whose intersection has dimension at least $d-t-1$. This notion interpolates between LC manifolds (which are the same as $1$-LC) and all manifolds (the same as $d$-LC); the case $t=d-1$ was also previously studied \cite{BBMogami}. In parallel, we introduce  ``$t$-constructible complexes'' as a generalization of constructible complexes, which correspond to the $t=1$ case. Intuitively, $t$-constructible $d$-complexes are defined recursively as those obtained by gluing two $t$-constructible $d$-complexes at a codimension-one subcomplex whose $(d-t)$-skeleton is constructible. 

With these two new properties, we prove the following generalization of the well-known result by Hochster  \cite{Hochster} that all constructible $d$-complexes are Cohen--Macaulay:

     \begin{thmnonumber}[Propositions \ref{prop:depth} and \ref{prop:CLC}]
    \label{theorem:MainTheoremII}
All $t$-constructible $d$-complexes have homotopical depth larger than $d-t$. Moreover, all $t$-constructible pseudomanifolds are $t$-LC.
    \end{thmnonumber}
    
    The converse of Theorem II is false, even if we restrict ourselves to $3$-manifolds. In fact, in \cite{BeLuKnots} there are two explicit examples (with 13 and 16 vertices, respectively) of two $3$-spheres containing a non-trivial knot that is realized by just three edges in their $1$-skeleton; the knots are the trefoil and the square knot, respectively. These examples have homotopical depth $3$ because they are spheres, and are $1$-LC by computation \cite{BeLuKnots}, but they are not $1$-constructible because of the knot \cite{HZ}.

\newpage    
\section{$t$-Constructible versus $t$-LC}
%
As in \cite{BBLc}, to which we refer for all definitions, we shall work with  \emph{simplicial regular CW-complexes}: These are finite regular CW-complexes where for every proper face $F$, the interval $[0,F]$ in the face poset of the complex is Boolean. 
The facets (i.e. the inclusion-maximal faces) of any simplicial regular CW-complex $K$ are therefore simplices; $K$ is \emph{pure} if all facets have the same dimension. Let $\sigma$ be a face of $K$. The \emph{star of $\sigma$ in $K$} is the subcomplex
    $\mathrm{St}(\sigma, K)=\{s\in K \mid \exists \tau \in K \text{ s. t. } \sigma \subset \tau \text{ and } s \subset \tau\}.$
    The \emph{link of $\sigma$ in $K$} is the subcomplex
    $\mathrm{link}(\sigma, K)=\{\tau\in \mathrm{St}(\sigma, K) \mid \tau\cap \sigma = \emptyset \}.$ 
The \emph{boundary of $K$} is the subcomplex
    $ \partial K = \{ s \in K  \mid \exists ! \tau \in K  \text{ s. t. } s \subsetneq \tau \}$.
    The faces of $K$ that do not belong to $\partial K$ are called \emph{interior}.
    If $K$ is a simplicial complex, link and boundary commute, in the sense that $\mathrm{link}(\sigma, \partial K) = \partial \mathrm{link}(\sigma, K)$ for all $\sigma$. 
    Moreover, if the dimension of $\sigma$ is $k$, for a $(d-1)$-face $\tau \in \mathrm{St}(\sigma, K)$ and the corresponding $(d-k-2)$-face $\tau' \in \mathrm{link}(\sigma, K)$, we know $\tau'$ is a boundary face of the $\mathrm{link}(\sigma, K)$ if and only if $\tau$ is a boundary face of $K$.
    
    By a \emph{d-pseudomanifold} we mean a finite regular CW-complex $P$ that is pure $d$-dimensional, simplicial, and such that each $(d -1)$-dimensional cell belongs to at most two $d$-cells. The boundary of the pseudomanifold $P$, denoted $\partial P$, is the smallest subcomplex of $P$ containing all the $(d -1)$-cells of $P$ that belong to exactly one $d$-cell of $P$. 
    According to this convention, adopted in \cite{BBLc}, a pseudomanifold needs not be a simplicial complex; it might be disconnected; and its boundary might not be a pseudomanifold, as shown by a cone over disjoint segments. 
    A \emph{tree of $d$-simplices} is a triangulation of the $d$-dimensional ball whose dual graph is a tree.
    From now on, we use the word ``faces'' as synonymous of ``cells''. We also adopt the convention that the empty set is a face of dimension~$-1$.

  \begin{definition}[$t$-LC]  \label{def:tLC}
  Let $d>1$ be an integer. Let $t \in \{1, \ldots,  d\}$. We call \emph{$t$-LC pseudomanifolds}  the $d$-dimensional pseudomanifolds obtainable from a tree of $d$-simplices by recursively identifying two boundary $(d-1)$-faces whose intersection has dimension at least $d-1-t$.
    \end{definition}
 
   A \emph{$t$-LC gluing } in the boundary of a pseudomanifold is the identification of two boundary facets $\Delta$ and $\Delta'$ whose intersection is at least $(d-t-1)$-dimensional.     The glued facets become interior, so they are not available for further gluings. A  \emph{$t$-local construction} for a pseudomanifold $M$  is a sequence of  $t$-LC gluings that obtains $M$ from some tree of $d$-simplices $T$. From Definition~\ref{def:tLC} it is clear that all $t$-LC pseudomanifolds are also  $(t+1)$-LC.  Three values of $t$ have already been studied in the literature:
\begin{compactitem}
\item    For $t=1$, the ``1-LC'' notion is the same as the ``LC'' notion in \cite{Durhuus} and \cite{BBLc};
\item For $t=d-1$, ``$(d-1)$-LC pseudomanifolds'' are the same as the ``Mogami pseudomanifolds'' introduced in \cite{BBMogami} and named after \cite{Mogami};
\item For $t=d$, all $d$-dimensional strongly-connected pseudomanifolds are $d$-LC. 
  \end{compactitem}
Note that there is a ``big jump'' from $t=d-1$ to $t=d$:  Any non-simply-connected manifold is an example of a $d$-LC pseudomanifold that is not $(d-1)$-LC.

Recall that \emph{constructible complexes} are defined inductively in the following way:
\begin{compactitem}
\item every simplex, and every $0$-complex, is constructible;
\item a $d$-dimensional pure simplicial complex $C$ that is not a simplex is constructible if and only if it can be written as  $C=C_1 \cup C_2$, where $C_1$ and $C_2$ are constructible $d$-complexes, and $C_1 \cap C_2$ is a (pure) constructible $(d-1)$-complex.
\end{compactitem}
Note that when $d=1$, `constructible' is synonymous with `connected'.

\begin{definition}[$t$-constructible] Let $t \le d$ be positive integers. We define \emph{$t$-constructibility} for $d$-dimensional simplicial complexes recursively, as follows:
\begin{compactitem}
\item every simplex is $t$-constructible;
\item a $1$-dimensional complex is $t$-constructible if and only if it is connected;
\item a $d$-dimensional pure simplicial complex $C$ that is not a simplex is $t$-constructible if and only if it can be written as  $C=C_1 \cup C_2$, where $C_1$ and $C_2$ are $t$-constructible \phantom{i} {$d$-complexes}, and $C_1 \cap C_2$ is a pure $(d-1)$-complex whose $(d-t)$-skeleton is constructible.
\end{compactitem}
Note that when $t=1$, ``$1$-constructible'' is synonymous with ``constructible''. 
\end{definition}

\begin{example} Inside the boundary of some shellable $3$-ball $B$, pick two triangles that belong to different tetrahedra in $B$ and that share exactly one vertex $v$. Let $X$ be the  subcomplex of $\partial B$ formed by these two triangles. Glue together two identical copies $B', B''$ of $B$ by identifying the corresponding subcomplexes $X' \equiv X''$. Let $P$ be the resulting $3$-dimensional pseudomanifold. The  link of $v$ in $P$ is topologically an annulus, so it cannot be constructible: in fact, it is easy to see from the recursive definition that the only constructible $2$-manifolds with non-empty boundary are triangulated $2$-disks. Hence, $P$ is not constructible, because constructibility is closed under taking links, cf.~\cite[section 11.2]{Basic}. Yet $P$ is $2$-constructible, because by taking $C_1 = B'$ and $C_2 = B''$ in the definition above, their intersection is a pure $2$-complex $X$ whose $1$-skeleton is a connected graph. 
\end{example}

\begin{remark}
It is an open question whether all $k$-skeleta of constructible complexes are themselves constructible. Should the answer to this problem be positive, then it is easy to see, by induction, that if $t<d$ all $t$-constructible $d$-complexes are also $(t+1)$-constructible.  
\end{remark}

\noindent Recall that a (pure) $d$-dimensional complex $C$ is called 
\begin{compactitem} [ -- ]
\item \emph{homotopy-Cohen--Macaulay}  if for any face $F$, for all $i  <  \dim \operatorname{link}(F,C)$, $\pi_i ( \operatorname{link}(F,C))=0$;
\item  \emph{Cohen--Macaulay}, if for any face $F$, for all $i < \dim \operatorname{link}(F,C)$, $H_i ( \operatorname{link}(F,C))=0$.
\end{compactitem}
By Hurewicz' theorem, the first notion is stronger. By choosing $F= \emptyset$, one sees immediately that all the homotopy groups  from the $0$-th to the $(d-1$)-st of any homotopy-Cohen--Macaulay complex must be  trivial. (And analogously, all homology groups  from the $0$-th to the $(d-1$)-st of Cohen--Macaulay complexes are trivial.) Thus any homology-sphere that is not simply-connected, like the 16-vertex triangulation in \cite{BjLu}, is an example of a Cohen--Macaulay complex that is not homotopy-Cohen--Macaulay.
Recall also that the \emph{homotopical depth} of a simplicial complex $C$ is defined as the maximum $k$ such that the $k$-skeleton of $C$ is homotopic-Cohen--Macaulay \cite[Chapter 3.6.1]{Jonsson}. The  \emph{(homological) depth} of $C$ is  defined algebraically, but it can be characterized  as the maximum $k$ such that the $k$-skeleton of $C$ is Cohen--Macaulay \cite[Theorem 4.8]{Smith}.  

Hence, the next Proposition generalizes Hochster's result that all constructible complexes are (homotopy) Cohen--Macaulay \cite{Hochster}:

 \begin{definition}[$t$-homotopy-CM, $t$-CM] Let $t \le d$ be positive integers.
A  (pure) $d$-complex $C$  is called:
\begin{compactitem}[ -- ]
\item  \emph{$t$-homotopy-CM} if it has homotopical depth $> d-t$;
\item \emph{$t$-CM} if it has depth $> d-t$.
\end{compactitem} 
Note that ``$1$-homotopy-CM'' and ``$1$-CM'' are synonymous with ``homotopy-Cohen--Macaulay'' and ``Cohen--Macaulay'', respectively.
 \end{definition}

\begin{proposition} \label{prop:depth}  Let $t \le  d$ be positive integers. 
For $d$-dimensional simplicial complexes, we have the following hierarchy: 
\[ \{ \textrm{$t$-constructible} \} \; \subset \; 
\{ \textrm{$t$-homotopy-CM} \} \; \subset \; 
\{ \textrm{$t$-CM} \}. \]
\end{proposition}

\begin{proof} The claim for $t=1$ (and arbitrary $d$) is a well-known result by Hochster \cite{Hochster}, so  we shall focus on the case $2 \le t \le d$. 
Also, the second inclusion is straightforward, so we shall focus on the first one.
Let $C$ be a   $t$-constructible $d$-complex with $N$ facets. If $C$ is a simplex, then it has homotopical depth $d$, and $d \ge d-t+1$. If $C$ is not a simplex, then $C=C_1 \cup C_2$, where $C_1$ and $C_2$ are $t$-constructible $d$-complexes, and $C_1 \cap C_2$ is a pure $(d-1)$-dimensional complex whose $(d-t)$-skeleton is constructible. By the inductive assumption with respect to $N$, the $(d-t+1)$-skeleta of $C_1$ and of $C_2$ are homotopy-Cohen--Macaulay. Since we are in the case $t \ge 2$, it is easy to see that their intersection is
\[ U \ \eqdef \ (d-t+1)\operatorname{-skel} \; (C_1 \cap C_2).\]
But by assumption, the $(d-t)$-skeleton of $C_1 \cap C_2$ is constructible, hence homotopy-Cohen--Macaulay, and it coincides with the $(d-t)$-skeleton of $U$. Hence, $U$ has dimension $d-t+1$ and homotopical depth $\ge d-t$. It follows 
(see  e.g.~Jonsson \cite[Lemma 3.32]{Jonsson} for the homotopical statement and  Hibi \cite[p.~98]{Hibi} for the homological one) 
that the union of the $(d-t+1)$-skeleta of $C_1$ and $C_2$ is homotopy-Cohen--Macaulay. In other words, both the homotopical depth  and the homological depth of $C_1 \cup C_2$ are at least $d-t+1$.
\end{proof}

\begin{lemma} \label{lem:skeleton} Let $t \le d$ be positive integers.\\
\emph{ (i)} If a $d$-complex is $t$-constructible, for all $0 \le k \le d$ its $k$-skeleton is strongly connected. \\
\emph{(ii)}  If a $d$-pseudomanifold is $t$-LC, for all $0 \le k \le d$ its $k$-skeleton is strongly connected.
\end{lemma}

\begin{proof}
(i): By induction on the number of $d$-faces. If $C$ is a simplex, the $k$-skeleton of $C$ is even shellable \cite{BjWa}. If not, then $C=C_1 \cup C_2$, where $C_1$ and $C_2$ are $t$-constructible $d$-complexes, and $C_1 \cap C_2$ is a pure $(d-1)$-dimensional complex whose $(d-t)$-skeleton is constructible. By inductive assumption, the $k$-skeleta of $C_1$ and of $C_2$ both have connected dual graphs. Moreover, any $(k-1)$-face $\sigma$ of $C_1 \cap C_2$ is contained in some $d$-face of $C_i$, and in particular in some $k$-faces of $C_i$, for $i=1,2$. Therefore there is an edge ``across $\sigma$'' in the dual graph of $C$ that connects the dual graphs of the $k$-skeleta of $C_1$ and of $C_2$.\\
(ii): The $k$-skeleton of a tree of simplices is strongly-connected. Any subsequent boundary gluing does not destroy this property.
\end{proof}

\begin{lemma} \label{lem:glue}  Let $t \le d$ be positive integers. Let $C$ be a $d$-dimensional pseudomanifold. If $C$ can be split in the form $C=C_1 \cup C_2$, where $C_1$ and $C_2$ are $t$-LC $d$-pseudomanifolds and $C_1 \cap C_2$ 
is a pure $(d-1)$-complex with strongly-connected $(d-t)$-skeleton, then $C$ is $t$-LC. 
\end{lemma}

\begin{proof}
The proof is a direct generalization of that of \cite[Lemma 2.23]{BBLc}. Fix a $t$-local construction for $C_1$ and $C_2$, and call $T_i$ the tree along which $C_i$ is constructed. Pick a $(d-1)$-face $\sigma$ in $C_1 \cap C_2$, which thus specifies a $(d-1)$-face in the boundary of $C_1$ and of $C_2$. Let $C'$ be the pseudomanifold obtained by identifying the two copies of $\sigma$.  Clearly, $C'$ has a $t$-local construction along the tree obtained by joining $T_1$ and $T_2$ by an edge across $\sigma$: Just redo the same $t$-LC gluings of the $C_i$'s. Now if $C_1 \cap C_2$ consists of only one simplex, then $C = C'$ and we are done. Otherwise, by the assumption, we can label the facets of $C_1 \cap C_2$ by $\sigma_0, \ldots, \sigma_m$, so that $\sigma_0 = \sigma$ and for each $k \ge 1$ the facet $\sigma_k$ intersects the union $\sigma_0 \cup \ldots \cup \sigma_{k-1}$ in a subcomplex of $\partial \sigma_k$ of dimension $\ge d-t-1$. Now for each $i$, identify the two copies of the facet $\sigma_i$ inside $C'$. All these gluings are $t$-LC, and eventually yield $C$.
\end{proof}

\begin{corollary} \label{cor:depth}  Let $C$ be a $d$-pseudomanifold. If $C=C_1 \cup C_2$, where $C_1$ and $C_2$ are $t$-LC $d$-pseudomanifolds and 
$C_1 \cap C_2$ 
is a pure $(d-1)$-complex of depth $\ge d-t$,  then $C$ is $t$-LC. 
\end{corollary}

\begin{proof}
The $(d-t)$-skeleton of $C_1 \cap C_2$ is Cohen--Macaulay, hence strongly-connected.
\end{proof}

\begin{proposition} \label{prop:CLC} Let $t \le d$ be positive integers. For $d$-dimensional pseudomanifolds, we have the following hierarchy: 
\[ \{ \textrm{$t$-constructible} \} \; \subset \; 
\{ \textrm{$t$-LC} \}  \;  \subset \;
\{ \textrm{ all } \}, 
\]
and for $d \ge t+2$ all inclusions are strict. 
\end{proposition}

\begin{proof} 
By Lemma \ref{lem:skeleton}, the $(d-t)$-skeleton of any $t$-constructible $(d-1)$-complex is strongly-connected. So Lemma \ref{lem:glue} (or Corollary  \ref{cor:depth}, together with Proposition \ref{prop:depth}) immediately implies by induction that all $t$-constructible complexes are $t$-LC.  To show the strictness of this inclusion,  take $(d-t-2)$ consecutive cones over the example constructed in Proposition \ref{prop:Linkt}, part (ii), and apply Proposition \ref{prop:Linkt}, part (i), and Proposition \ref{prop:cone} below. Finally, the second inclusion is obvious, and its strictness (already for $d \ge t+1$) will be shown in Corollary \ref{cor:hierarchy} below.
\end{proof}

\begin{proposition} \label{prop:Linkt} Let $t$ be any positive integer. \\
\emph{  (i) } For any $d > t$, the link of any vertex in a $t$-constructible $d$-complex, is $t$-constructible. \\
\emph{ (ii) } There exists a  $(t+2)$-dimensional pseudomanifold $M$ that is $t$-LC, but has a vertex link that is neither $t$-LC nor $t$-constructible.
\end{proposition}

\begin{proof} (i): 
We proceed by induction on the number of $d$-faces. Let $C$ be a $t$-constructible $d$-complex and let $v$ be a vertex of $C$. If $C$ is a simplex, the claim is clear. If not, then $C = C_1 \cup C_2$, where $C_1$ and $C_2$ are $t$-constructible $d$-complexes, $C_1 \cap C_2$ 
is a pure $(d-1)$-complex, and  $S : = (d-t)\operatorname{-skel}( C_1 \cap C_2)$ is constructible. If $v$ is not in $C_1$, then 
$\operatorname{link} (v,C) =  \operatorname{link} (v,C_2)$,
so the claim follows by inductive assumption, since $C_2$ has fewer facets. The case $v \notin C_2$ is symmetric. So without loss, we may assume $v \in C_1 \cap C_2$.
Set $L := \operatorname{link} (v,C)$ and $L_i := \operatorname{link} (v,C_i)$ ($i=1,2$). 
Clearly $L = L_1 \cup L_2$. By induction, both $L_1$  and $L_2$  are $t$-constructible. It is easy to see that 
$L_1 \cap L_2 =  \operatorname{link} (v,C_1 \cap C_2)$
is pure $(d-2)$-dimensional. Thus to conclude that $L$ is $t$-constructible, we need to show that the $(d-1-t)$-skeleton of $L_1 \cap L_2$ is constructible. Since constructibility is closed under taking links, and since $S$ is constructible, it suffices to show that
\begin{equation} \label{claim}
(d-1-t)\operatorname{-skel}( L_1 \cap L_2) \ = \ \operatorname{link}(v,S).
\end{equation}
Let us prove relation (\ref{claim}). If $\sigma$ is in the left hand side, then 
there is a $(d-2)$-face $F$ of $L_1 \cap L_2$ containing $\sigma$. The $(d-1)$-face $v \ast F$ is thus a 
facet of $C_1 \cap C_2$.  Since $v \ast F$ contains $v \ast \sigma$, it follows that $v \ast \sigma$ is in $S$, so $\sigma$ is in $\operatorname{link}(v,S)$. Conversely, if $\sigma$ is in $\operatorname{link}(v,S)$, then $v \ast \sigma$ is contained in some $(d-1)$-face $F$ of $C_1 \cap C_2$. But then the $(d-2)$-face $F' = F \setminus \{v\}$ contains $\sigma$ and belongs to $L_1 \cap L_2$. So $\sigma$ is contained in the $(d-1-t)$-skeleton of $L_1 \cap L_2$.  \\
 (ii): Let $H_i$ be the $(t+1)$-simplex of consecutive vertices $[i, i+1, \ldots, i+t+1]$. Let $m \ge t+3$. Consider the ``tree of $m$ $d$-simplices''
\[ H_1, H_2, \ldots, H_m \] 
and let  $P$ be the simplicial complex obtained from it by gluing together the first and the last vertex, i.e., vertices $1$ and $m+t+1$. Clearly, $P$ is strongly-connected, but the link of $1$ in $P$ consists of two disjoint $t$-simplices. Now choose any shellable $(t+2)$-ball $B$ that contains in its boundary a copy of $P$. Glue together two copies of $B$ by identifying the corresponding copies of~$P$, and call $M$ the resulting $(t+2)$-dimensional pseudomanifold. By Lemma \ref{lem:glue}, $M$ is $t$-LC, because the $t$-skeleton of $P$ is strongly-connected. However, the link of $1$ inside $M$ consists of two $(t+1)$-balls glued together at a disjoint union of two  $t$-simplices: In other words, the link of $1$ inside $M$ is homeomorphic to $\mathbb{S}^1 \times \mathbb{I}^t$, and thus homotopy equivalent to  a $1$-sphere. Since it is not simply-connected, the link of $1$ inside $M$ is not $t$-LC, and in particular not $t$-constructible. In fact, by Proposition \ref{prop:depth}, any triangulation of the link of $1$, being homeomorphic to $\mathbb{S}^1 \times \mathbb{I}^t$, is not $t$-constructible. In particular, $M$ (and any subdivision of it) cannot be $t$-constructible, by Part (i) above.
\end{proof}


\begin{proposition} \label{prop:cone} Let $t \le d$ be positive integers. Let $C$ be a $d$-pseudomanifold. Let $v$ be a new point. 
$C$ is $t$-LC if and only if $v \ast C$ is $t$-LC.
\end{proposition}

\begin{proof} The ``only if'' direction is easy: Suppose $C$ is obtained from a tree of $d$-simplices $T$ with a sequence of $t$-LC gluings, where the $i$-th gluing identifies faces $F_i$ and $F'_i$ with intersection of dimension $\ge d-1-t$. Then $v \ast C$ is obtained from the tree $v \ast T$ of $(d+1)$-simplices with the sequence of gluings that at the step $i$ glues together $v \ast F_i$ and $v \ast F'_i$; and the intersection 
\[(v \ast F_i) \cap (v \ast F'_i) =  v \ast (F_i \cap F'_i)\]
has dimension $\ge d-t$, so all these steps are legitimate $t$-LC gluings.

The ``if'' direction is perhaps more surprising, because as we saw in Proposition  \ref{prop:Linkt}, the $t$-LC property is not maintained by links.
Yet a similar argument of \cite[Proposition 3.25]{BBLc} works.  Suppose $v \ast C$ is $t$-LC and  let $T_i \rightarrow T_{i+1}$ be any step in some $t$-local construction of $v \ast C$. This step glues two $d$-faces $F$ and $G$ of $\partial T_i$ sharing a $(d-t)$-face $\sigma$. Since $F$ and $G$ will end up in the interior of $v\ast C$, both contain a copy of $v$, since $C \subset \partial (v\ast C)$. If $\sigma$ contains $v'$, a copy of $v$, then by gluing $F=v'\ast F'$ with $G=v'\ast G'$, we glue $(d-1)$-faces $F'$ and $G'$ sharing the $(d-t-1)$-face $\tau$, where $\sigma = v'\ast \tau$. But if $\sigma$ does not contain a copy a $v$, then $\sigma \in C$. Gluing $F$ and $G$ corresponds to possibly many gluings of $(d-1)$-faces $F'$ and $G'$, where $F', G'$ share $\sigma$ and do not contain any copy of $v$. Hence performing these gluings one by one, following the $t$-local construction of $v\ast C$, we eventually obtain a $t$-local construction for $C$.
\end{proof}

\begin{corollary} 
\label{cor:tLCsuspensions} Let $t \le d$ be positive integers. If a $d$-pseudomanifold is $t$-LC, its suspension is $t$-LC.
\end{corollary}

\begin{proof} Let $A$ be a $t$-LC $d$-pseudomanifold. Let $v_1, v_2$ be two new vertices. Let $C_i = v_i \ast A$.  By Proposition \ref{prop:cone}, each $C_i$ is $t$-LC. Glue $C_1$ and $C_2$ together by identifying the two copies of $A$. By Lemma \ref{lem:skeleton}, $A$ has strongly-connected $(d+1-t)$-skeleton. By Lemma \ref{lem:glue}, $C_1 \cup C_2$, which is the suspension of $A$, is $t$-LC. 
 \end{proof}

\begin{corollary} \label{cor:hierarchy} For any $d \ge 2$, for any $1 \le t \le d-1$, there exists a $d$-pseudomanifold that is $(t+1)$-LC, but not $t$-LC.
\end{corollary}

\begin{proof}
By induction on $d$. For $d=2$, any surface different from the sphere is $2$-LC, but not $1$-LC \cite{Durhuus}. 
For $d \ge 3$: If $t=d-1$, any non-simply-connected $d$-manifold is $(t+1)$-LC, but not $t$-LC. If instead $t \le d-2$, by the inductive assumption there is a  $(d-1)$-pseudomanifold that is $(t+1)$-LC, but not $t$-LC; coning over it, by Proposition \ref{prop:cone} we conclude. 
\end{proof}

\begin{remark}
The pseudomanifold constructed in the previous corollary is not a manifold. At the moment, we do not know an explicit example of a $d$-manifold that is $2$-LC but not $1$-LC. (See also Example \ref{ex:nonLC} below.) A good candidate for $d=5$ might be the double suspension $S$ of the $16$-vertex Poincar\'e homology sphere by Bj\"orner and Lutz \cite{BjLu}. This $S$ is not $1$-constructible, because it is not PL; some experiments with the Random Discrete Morse algorithm \cite{BeLu} seem to suggest that $S$ is likely not $1$-LC either. On the positive side, we do know that $S$ is (at most) $3$-LC, in view of Corollary~\ref{cor:tLCsuspensions}, part (ii), applied twice. 
\end{remark}

\begin{corollary} For any $d \ge 3$, not all triangulated $d$-balls are $2$-LC.
\end{corollary}

\begin{proof} In view of Proposition \ref{prop:cone}, we only need to construct a $3$-ball that is not $2$-LC (or equivalently, not Mogami), a task that was already carried out in \cite{BBMogami}. 
\end{proof}

\begin{remark}
For any $t \le d$, being  $t$-LC is a property that is algorithmically recognizable, simply by trying all possible ``spanning trees of $d$-simplices'' and all possible boundary matchings. 
Moreover, for any $t \le d-1$, being $t$-LC implies being simply connected. Thus in view of the Poincar\'e conjecture (proven for $d \ge 5$  by Smale \cite{Smale} and for $d=4$ by Freedman \cite{Freedman}), if $M$ is a $d$-manifold with the homology of a sphere, the fact that  $M - \Delta$ is $t$-LC  for some facet $\Delta$ implies that $M$ is a $d$-sphere. So were all $d$-balls $t$-LC for some $t \le d-1$, then for any manifold $M$  we could decide if $M$ is a $d$-sphere or not just, first by checking whether $M$ has the homology of a sphere,  and then by checking whether $M$ minus some facet is $t$-LC. 
In conclusion, S.~P.~Novikov's theorem on the algorithmic unrecognizability of $d$-spheres for $d \ge 5$ ~\cite{VKF} implies that for every $d \ge 5$ and every $t \le d-1$, there must exist $d$-balls that are not $t$-LC. It is conjectured that also $4$-balls are not algorithmically recognizable, which would imply in the same non-constructive way that some $4$-balls are not $3$-LC.
\end{remark}

\section{An exponential bound for 2-LC manifolds}
Since $1$-LC $d$-manifolds are exponentially many, while $d$-LC $d$-manifolds (which is the same as saying ``all $d$-manifolds'') are more than exponentially many,  a natural question is whether one can give exponential bounds  for $t$-LC $d$-manifolds also for some $t$ larger than $1$. In this section we realize a first step in this direction: For fixed $d \geq 3$, we prove that there are less then $2^{\frac{d^3}{2} N }$ combinatorially distinct simplicial 2-LC $d$-manifolds with $N$ facets (Theorem \ref{thm:main}). 

\subsection{Excluding some gluings}
\label{sec:Gluings_and_Link_Transformations}
First we establish which 2-LC gluings can actually lead to a manifold without boundary.  
    \begin{lemma}
    \label{lemma:gluings}
        Only the $2$-LC gluings satisfying all the conditions below can lead to a triangulated manifold without boundary:
        \begin{compactenum}[ \rm (i)]
            \item preserving orientability of links of $(d-3)$-faces (Figure \ref{fig:GL111}),
            \item planar with respect to the involved links of $(d-3)$-faces (Figure \ref{fig:GL2ndBad1}),
            \item impacting only on the boundaries of the links of $(d-3)$-faces.
        \end{compactenum}
    Moreover, the number of ways we can glue the boundary facets to one another is completely determined by looking at edges of the boundaries of links of $(d-3)$-faces.
    \end{lemma}

    \begin{figure}[hbt] 
    \centering
		\subfloat[This gluing makes the link non-orientable, so the link will never be homeomorphic to $S^2$.  ]{\includegraphics[width=12em]{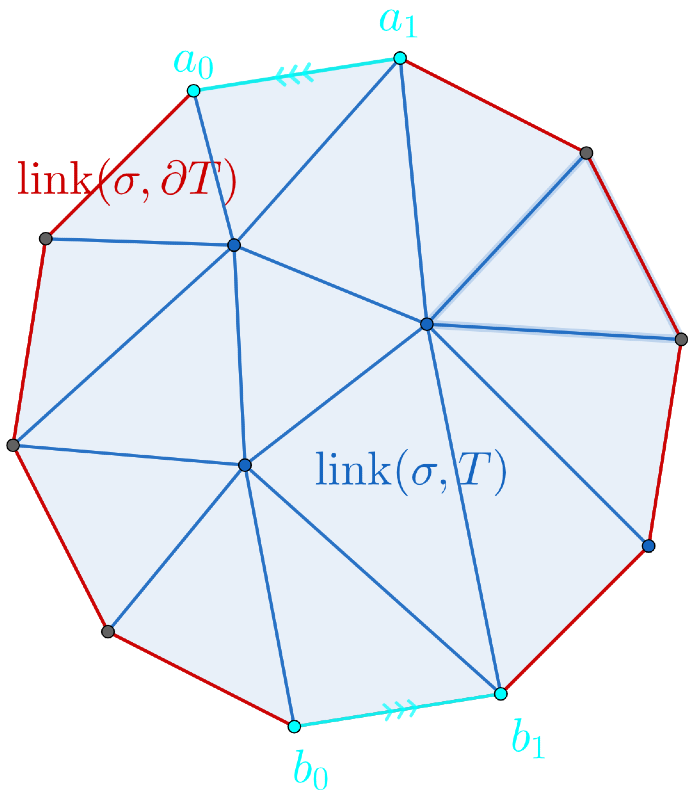} 
		\label{fig:GL111}}
		\hfill
		\subfloat[A non-planar gluing, by the Jordan-Schoenflies theorem,  makes the final link not homeomorphic to $S^2$. ]{\includegraphics[width=16em]{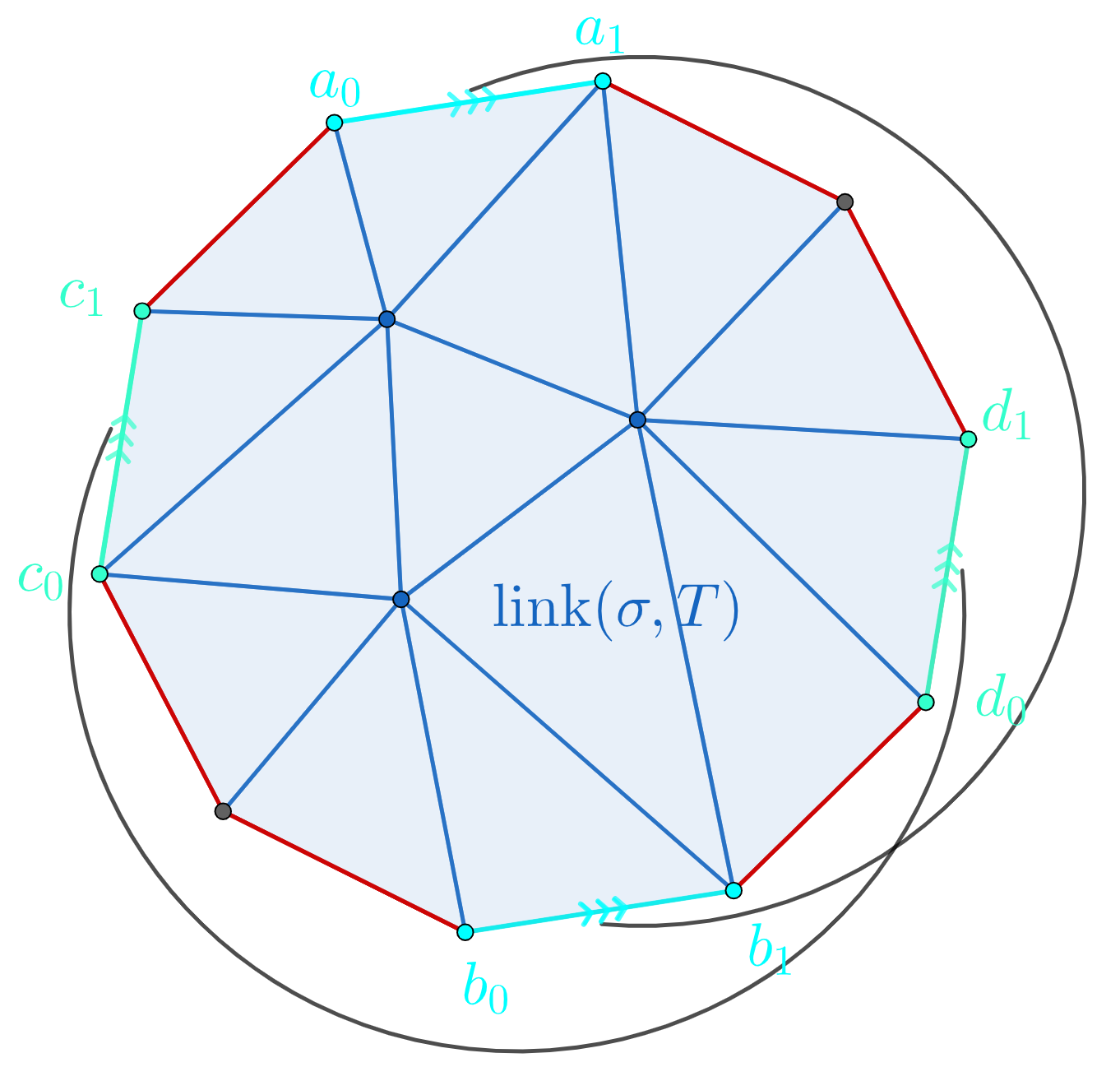}
		\label{fig:GL2ndBad1}}
		\caption{Ways of gluing that do not lead to a manifold without boundary.}
		\label{fig:GL12}
    \end{figure}

     \begin{proof}
         Let $T$ be a tree of $d$-simplices. 
         The link of a $(d-3)$-face $\sigma$ in $T$ is a triangulated disk, whose boundary circle is the boundary link of $\sigma$. Consider two boundary facets $E_1$ and $E_2$ sharing $\sigma$, and corresponding to edges $[a_0,a_1] \sim [b_0,b_1]$ in $\mathrm{link}(\sigma, T)$. There are two options for such an identification: $a_0 \sim b_1$ and $a_1 \sim b_0$, or $a_0 \sim b_0$ and $a_1 \sim b_1$.  As displayed in Figure \ref{fig:GL111}, the first option makes the link non-orientable, which cannot be fixed by any gluing. Also, if we further glue other two boundary facets $E_3$ and $E_4$ with respect to $\sigma$, 
      it is clear that any   ``non-planar matching'' creates a contradiction with the Jordan--Schoenflies' theorem, cf. Figure \ref{fig:GL2ndBad1}.
        So (i) and (ii) are clear.
        Now, any gluing $E_1 \sim E_2$ can affect the link of some other $(d-3)$-face $\delta$ in two possible ways: 
 \begin{compactitem}       
    \item If $\delta$ is contained in $E_i$ for some $i=1,2$, then there is a $(d-3)$-face $\delta'$ such that $\delta \sim \delta'$ as a consequence of $E_1 \sim E_2$. Suppose by contradiction that $\mathrm{link}(\delta, T)$ and $\mathrm{link}(\delta', T)$ are, or become, connected by an edge $[a_0,a_1]$ which is an interior edge of one of these links, say of $\mathrm{link}(\delta, T)$. Then $[a_0,a_1]$ corresponds to an interior $(d-1)$-face $F_1$ of $T$ in $\mathrm{St}(\delta, T)$ and it also corresponds to a $(d-1)$-face $F_2$ of $T$ in $\mathrm{St}(\delta', T)$. After the identification, $F_1$ and $F_2$ share the $(d-3)$-face $\delta\sim\delta'$ and the edge $[a_0,a_1]$. In order to end up with a simplicial complex, we need to identify $F_1$ and $F_2$ at some point, which is not possible since $F_1$ is an interior $(d-1)$-face. 
    En passant, note that $\mathrm{link}(\delta, T)$ and $\mathrm{link}(\delta', T)$ become connected with at least one boundary edge, since they are contained in $E_i$. Suppose the links are, or become, connected by a vertex $v$ that is an interior vertex of one of the original links, say of $\mathrm{link}(\delta, T)$. In that case, we create an $S^1$ just around $v$ (containing no other vertex) in the interior of $\mathrm{link}(\delta, T)$. In the link of $\delta \sim \delta'$ in the new complex, $v$ appears on both sides of the $S^1$, because now the two original links are connected by a boundary edge and at the same time no interior triangle or interior edge is identified. Which is a contradiction.
    \item If $\delta$ is contained in any $E_i$, it is not identified with another $(d-3)$-face. Here, it may happen that an edge or a vertex gets identified within the link itself. We can use the same reasoning as for the first case to conclude that no identifications can happen in the interior of the link.
 \end{compactitem}
   The arguments above work not only for a tree of $d$-simplices $T$, but also for any pseudomanifold obtained from $T$ by performing 2-LC gluings that satisfy the conditions of this lemma. 
   After connecting links or after an identification within one link, an interior edge of the link of a $(d-3)$-face can only correspond to
 \begin{compactitem}       
    \item either one interior $(d-1)$-face (if it was an interior edge of the link already),
    \item or two boundary facets (if it is a glued boundary edge of the link).
 \end{compactitem}
    All other options cannot lead to a manifold. 
    So the only possible issue  is an interior edge $[a_0,a_1]$ in the link of a $(d-3)$-face that corresponds to two distinct boundary facets. Assume $E_1$, $E_2$ are the two boundary facets in a pseudomanifold $M$ obtained from $T$ after some number of allowed 2-LC gluings; assume also that $E_1$ and $E_2$ are in $\mathrm{St}(\delta, \partial M)$, where $\delta$ is a $(d-3)$-face of $M$, and that $E_1$, $E_2$ share an edge $[a_0,a_1] \in \mathrm{link}(\delta, M)$. Now suppose we need to identify $[a_0,a_1]$ with another edge $[b_0,b_1]\in\mathrm{link}(\delta, M)$. The edge $[b_0,b_1]$ corresponds to (at least one) $(d-1)$-face of $\mathrm{St}(\delta, \partial M)$. Now we have three $(d-1)$-faces sharing $\delta$ and $[a_0,a_1] \sim [b_0,b_1]$. So to get a simplicial complex at the end, we need to identify all three of them, which is not possible. 
    In conclusion, no identification can be performed within interiors of links of $(d-3)$-faces. Which proves part (iii). 
    
  As for the final claim:  If there is an interior edge $[a_0,a_1]$ of $\mathrm{link}(\delta, M)$ that corresponds to two distinct boundary facets $E_1$ and $E_2$ in $\mathrm{St}(\delta, \partial M)$ for some $(d-3)$-face $\delta$ of $M$, then $E_1$ and $E_2$ share the face $\delta$ and $[a_0,a_1]$. Hence, they have to be identified at some point, in order to get a simplicial complex at the end of the 2-local construction. So the edge $[a_0,a_1]$ does not add to the number of ways in which one can glue boundary facets. Note that $E_1$, $E_2$ can be identified with respect to $\delta$ by a 2-LC gluing.
  
     \end{proof}

    The link of any $(d-3)$-face during a 2-local construction consists topologically of possibly punctured disks, connected by boundary edges and vertices to one another, or even to themselves. 

    \begin{lemma}
    \label{lemma:connected_components}
   The $2$-LC gluings that force an identification between two distinct connected components of the boundary of the link of a $(d-3)$-face do not lead to a manifold. 
    \end{lemma}

    \begin{proof}
        Denote $M$ a pseudomanifold obtained from a tree of $d$-simplices by performing 2-LC gluings that satisfy the conditions of Lemma \ref{lemma:gluings} and this lemma. Assume for a contradiction that we identify vertices $v_0$ and $v_1$ from two distinct connected components of $\partial \mathrm{link}(\delta, M)$, where $\delta$ is a $(d-3)$-face of $M$. We then create an $S^1$ in the interior of the link with the connected component containing $v_0$ fully inside this $S^1$, and the connected component containing $v_1$ fully outside. In the link of $\delta$ in the new complex, $v_0\sim v_1$ appears on both sides of the $S^1$, which is a contradiction using Jordan--Schoenflies' theorem as we need this link to eventually become $S^2$.
    \end{proof}

\subsection{Bounding the 2-LC gluings}

\begin{lemma}     Let $\sigma$ be a $(d-3)$-face in a tree $T$ of $N$ $d$-simplices. Let $m$ be the number of facets of $\partial T$ containing $\sigma$. Let $M(m)$ be the number of  ways we can glue those boundary facets among each other. Then we have the inequality
    \begin{equation}
    \label{eq:MmBound}
         M(m) \leq \ C_m \leq \ 4^m, 
    \end{equation}
    where $C_m= \frac{1}{m+1} \binom{2m}{m}$ is the $m$-th Catalan number. 
\end{lemma}

\begin{proof}
     When we glue  with respect to the $(d-3)$-face $\sigma$, we choose a boundary facet containing $\sigma$ and 
    \begin{compactenum}[ (a)]
      \item  either we leave it unidentified (Figure \ref{fig:Rec2}), which corresponds to choosing an edge of $\partial \mathrm{link}(\sigma, T)$ and `leaving it alone', 
     \item or we glue it with another boundary facet containing $\sigma$ (Figure \ref{fig:Rec1}), which corresponds to choosing two edges of $\partial \mathrm{link}(\sigma, T)$ and matching them.  
    \end{compactenum}
    \begin{figure}[hbt] 
    \centering
		\subfloat[]{\includegraphics[width=12em]{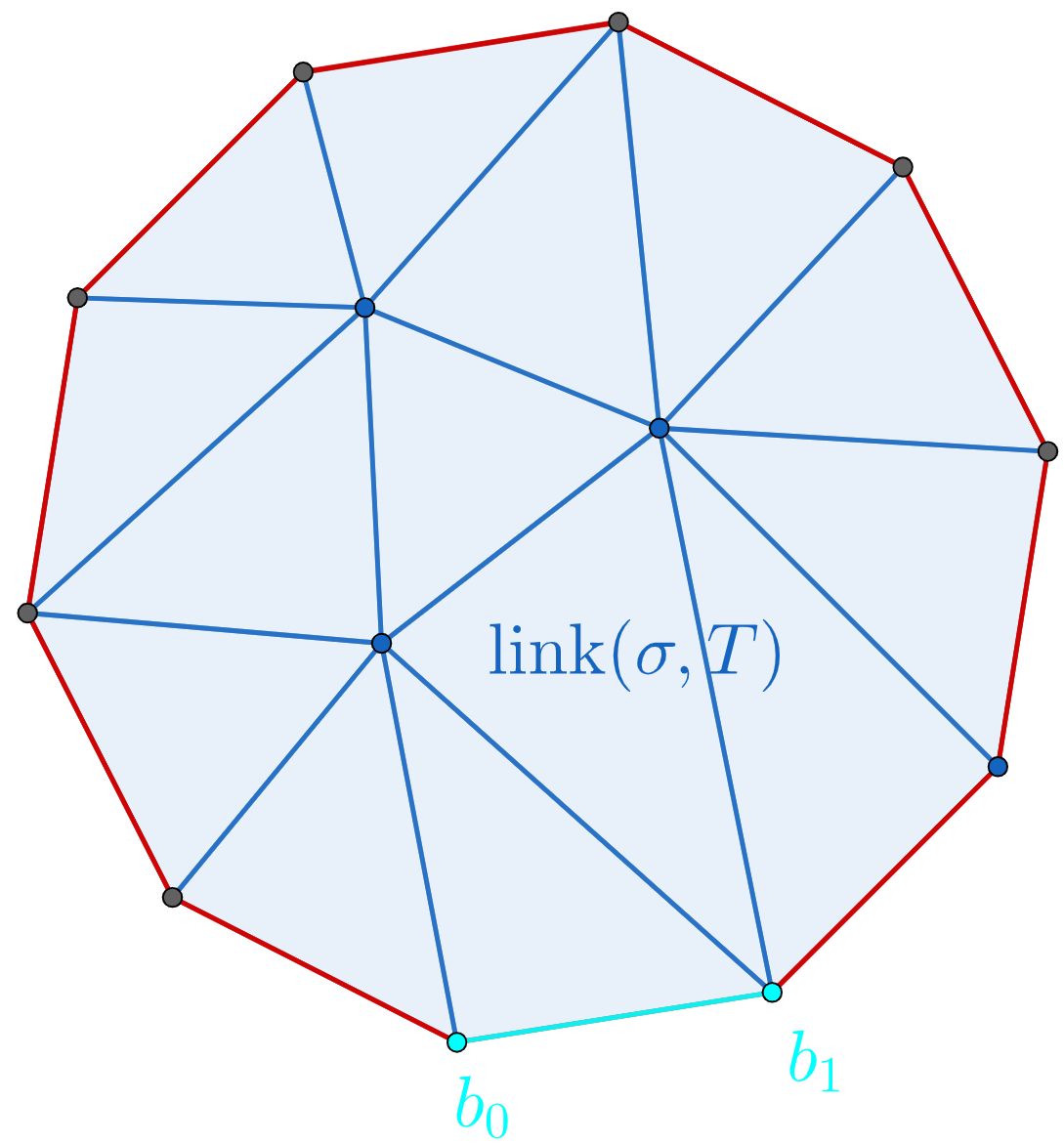}
		\label{fig:Rec2}}
		\hfill
		\subfloat[]{\includegraphics[width=12em]{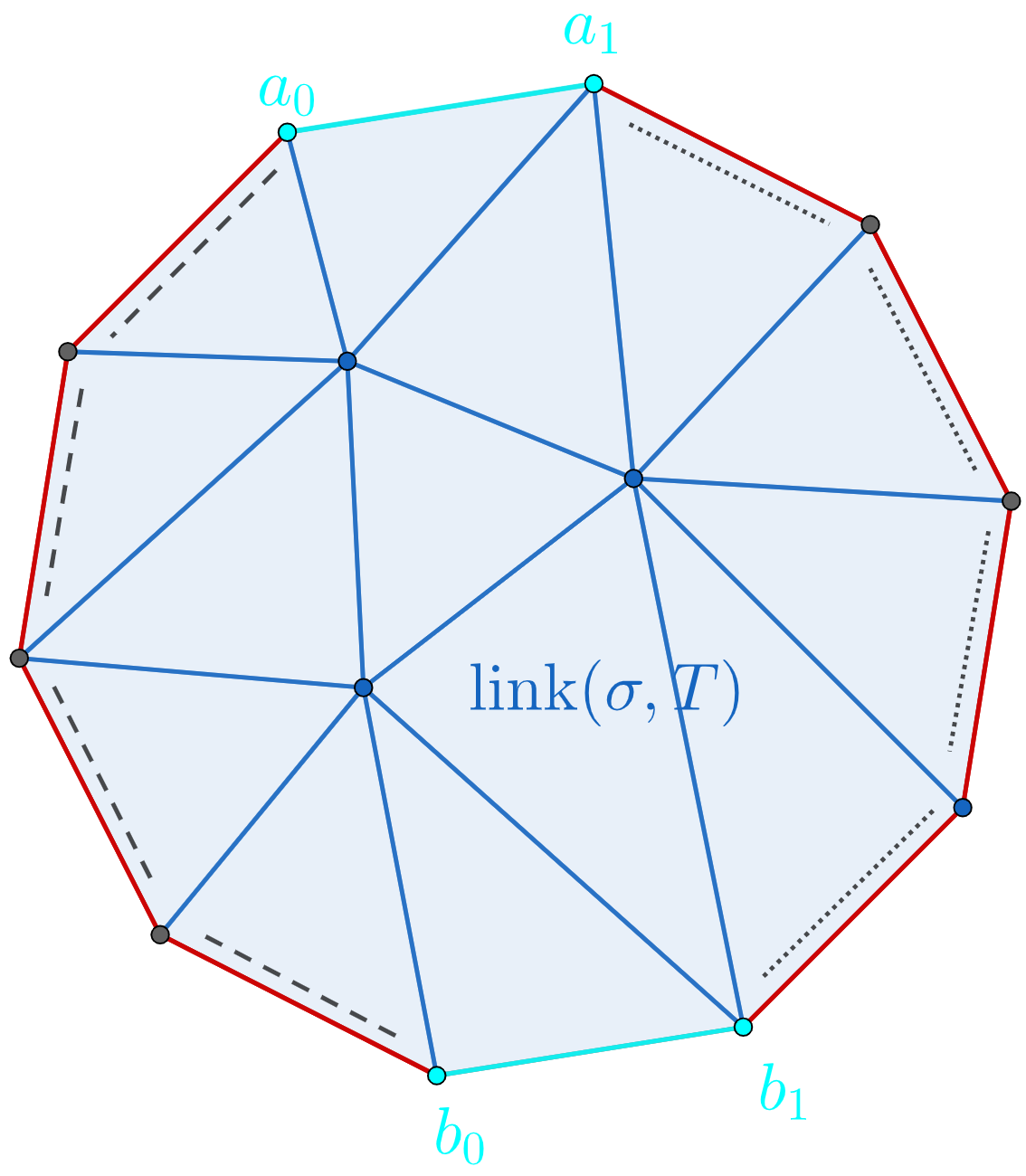}
		\label{fig:Rec1}}
		\caption{Two options for gluing in the very first step.}
		\label{fig:Rec12}
    \end{figure}
When an edge in the link is `left alone', this has no impact on how the remaining edges should be matched. Therefore, the number of gluing configurations with the first gluing of the type (a) is $M(m-1)$.
Instead, when we glue a pair of edges, we divide the boundary of the link of $\sigma$ into two parts (Figure \ref{fig:Rec1}): the first contains $m_1$ edges (dashed lines in the figure), the second $m_2$ edges (dotted lines in the figure), with $m_1+m_2 = m-2$. Since only planar gluings are allowed, the number of gluings with a first gluing of type (b) is thus $M(m_1)\cdot M(m_2)$. So
\begin{equation}   \label{eq:MCompareA}
    M(m) =  M(m-1) + \sum_{m_1+m_2=m-2} M(m_1)M(m_2).
\end{equation}
Set $M(0)=M(1)=1$. By definition, $M(m)\leq M(m+1)$.
Moreover
\begin{equation} \label{eq:MCompareB}
    \sum_{m_1+m_2=m-2} M(m_1)M(m_2)\ \ \leq \sum_{m_1+m_2=m-1} M(m_1)M(m_2) - M(m-1) \ .
\end{equation}
Putting Inequalities (\ref{eq:MCompareA}) and (\ref{eq:MCompareB}) together, we get by induction
\begin{equation}
    \label{eq_true2}
    M(m) \leq \sum_{m_1+m_2=m-1} M(m_1)M(m_2)  \: \:  \leq \: C_m \: < \: 4^m\ . \qedhere 
\end{equation}
\end{proof}

\begin{lemma}
    Consider an intermediate complex $K$ in a $2$-local construction. The number $M(m)$ bounds the number of possible gluing configurations among the boundary facets containing a given $(d-3)$-face of K. 
\end{lemma} 

\begin{proof}
Let $m$ be the number of boundary edges of $\mathrm{link}(\sigma, K)$. By Lemma \ref{lemma:connected_components}, we can partition the edges into sets $S_1$, $S_2$, $\dots$, $S_k$ such that each $S_i$ contains $m_i$ edges that can be glued among themselves but cannot be glued to any edge outside $S_i$. Since $\sum m_i=m$, by (\ref{eq:MmBound}) 
    \[
        M(m_1)M(m_2) \cdots M(m_k) \; \leq  \; 4^{m_1+m_2+ \ldots + m_k} \: = \: 4^{m}. \qedhere
\]
\end{proof}

\subsection{Proof of the Main Theorem}

\begin{lemma} \label{lem:DNrel} Let $N$ and $d$ be positive integers, $d \ge 3$. Let $T$ be a tree of $N$  $d$-simplices. Let $N_i$ be the number of boundary $i$-faces in $T$. Set 
$D:= 1 + \frac{N(d-1)}{2}$. Then
    \begin{equation} N_{d-1}=2D, \: \: N_{d-2}=dD, \: \: \textrm{ \em and } \: N_{d-3}= \frac{d}{6} (Nd^2+2N-3Nd+3d-3). 
    \label{eq:DNrel}
    \end{equation} 
    
\end{lemma}
\begin{proof}
    Note that all the the $i$-faces with $i\leq d-2$ are on the boundary of $T$. It is easy to see that $N_{d-1}=2+N(d-1)$, or in other words $N_{d-1}=2D$. By counting, or by \cite{BBLc}, we get
    \begin{equation*}
        N_{d-2}=\frac{d}{2} ( N(d-1)+2 ) =dD.
    \end{equation*}
    Now, any $d$-simplex contains $\binom{d+1}{d-2}=\frac{(d+1)d(d-1)}{6}$ $(d-3)$-faces, and any $(d-1)$-simplex contains $\binom{d}{d-2}=\frac{d(d-1)}{2}$ $(d-3)$-faces. So 
    \begin{equation*}
    \label{eq:dim3faces}
        N_{d-3}=N\frac{(d+1)d(d-1)}{6} - (N-1)\frac{d(d-1)}{2} = \frac{d}{6} (Nd^2+2N-3Nd+3d-3).
    \end{equation*}
\end{proof}

\begin{theorem} \label{thm:main}
    For any $d \geq 3$, the number of combinatorially distinct 2-LC $d$-manifolds 
    with $N$ facets, for $N$ large, is smaller than
\[ 
2^{\frac{d^3}{2}N} .
\]    
\end{theorem}

\begin{proof} By \cite[Corollary 4.3]{BBLc}, the number of trees with $N$ facets is at most 
\begin{equation}
\label{eq:numberOfTrees}
    \left( d \left(\frac{d}{d-1}\right)^{d-1} \right)^N < (de)^N\ .    
\end{equation}
    So we want to obtain an exponential upper bound for all manifolds obtainable from a given, fixed tree. Any tree $T$ of $N$ $d$-simplices has $2D= 2+N(d-1)$ boundary facets. Hence $D$ disjoint couples need to be glued together. Following \cite{Durhuus}, we partition any such perfect matching into rounds. The first round consists of couples that share a $(d-3)$-face in the boundary of the initial tree. Recursively, the $(i+1)$-st round consists of all couples that get to have common $(d-3)$-faces only after a gluing in the $i$-th round. 
    Denote by  $m_1$ (respectively, by $m_{1,\sigma}$) the number of the boundary facets glued in the first round (respectively, glued in the first round with respect to $\sigma$). Choose a set $\mathfrak{F}$ of $n_1$ $(d-3)$-faces, the ones with respect to which we shall perform the first round of gluings. We have
    \begin{equation*}
        \sum_{\sigma \in \mathfrak{F}} m_{1,\sigma} \: = \: m_1 \ .
    \end{equation*}
    Therefore, the number of possible gluing configurations in the initial tree, while fixing the $n_1$ $(d-3)$-faces in $\mathfrak{F}$, is bounded by 
    \begin{equation}
        \prod_{\sigma \in \mathfrak{F}} M(m_{1,\sigma}) \quad 
        \leq \prod_{\sigma \in \mathfrak{F}} 4^{m_{1,\sigma}} = 4^{m_1} \ .
    \end{equation}
    	There are $\binom{N_{d-3}}{n_1}$ ways to choose the $n_1$ $(d-3)$-faces in $\mathfrak{F}$. As $n_1 \leq N_{d-3}$, the number of possibilities in the first round is therefore at most
    \begin{equation}
    	\sum_{n_1=1}^{\frac{m_1}{2}} \binom{N_{d-3}}{n_1} 4^{m_1} < \sum_{n_1=0}^{N_{d-3}} \binom{N_{d-3}}{n_1} 4^{m_1} = 2^{N_{d-3}} 4^{m_1}.
    \end{equation}
Now for each couple in the first round, at most $\Big(\frac{d(d-1)}{2}-1 \Big)$ distinct $(d-3)$-faces become identified. So the total number of identified $(d-3)$-faces in the first round is at most
	\begin{equation*}
	    L_2:=\frac{m_1}{2}\left ( \frac{d(d-1)}{2}-1 \right ).    
	\end{equation*}
	We select $n_2$ out of these $L_2$ $(d-3)$-faces for the second round. The number of possible gluing configurations in the second round of gluings, while fixing the $n_2$ $(d-3)$-faces, is bounded by 
    \begin{equation}
        \prod_{\sigma \text{ is one the $n_2$ $(d-3)$-faces}} 4^{m_{2,\sigma}} = 4^{m_2} \ .
    \end{equation}
	The number of options for the second round is therefore at most 
    \begin{equation*}
        \sum_{n_2=1}^{L_2} \binom{L_2}{n_2} 4^{m_2} < \sum_{n_2=0}^{L_2} \binom{L_2}{n_2} 4^{m_2} = 2^{L_2} \, 4^{m_2}.
    \end{equation*}
    
    The same way, the number of possibilities in the $i$-th round, $i\geq 2$, is at most 
    \begin{equation}
        \sum_{n_i=1}^{L_{i}} \binom{L_{i}}{n_i} 4^{m_i} < 2^{L_{i}} \, 4^{m_i} , 
    \end{equation}
    where \[L_i := \frac{m_{i-1}}{2} \left( \frac{d(d-1)}{2}-1 \right).\]
    
    The following is therefore an upper bound on total number of 2-LC $d$-manifolds without boundary that can be constructed from a given tree:
   
    \begin{align}
        \notag
        \sum_{f=1}^{D} \quad & \sum_{\substack{m_1, m_2, \dots, m_f \\ \sum m_i = 2D \\ m_i \text{ even, } m_i\geq2 }} \ \left( \sum_{n_1=0}^{N_{d-3}} \binom{N_{d-3}}{n_1} 4^{m_1} \right) \left( \sum_{n_2=0}^{L_2} \binom{L_2}{n_2} 4^{m_2} \right) \cdots \left( \sum_{n_f=0}^{L_{f}} \binom{L_{f}}{n_f} 4^{m_f} \right) \\
        \notag
        &\leq \ \sum_{f=1}^{D} \quad \sum_{\substack{m_1, m_2, \dots, m_f \\ \sum m_i = 2D \\ m_i \text{ even, } m_i\geq2 }} 2^{N_{d-3}} \cdot  2^{\sum_{i=2}^{f} L_i} \cdot 4^{\sum_{i=1}^{f} m_i} \\
        \label{eq:bound3}
        &\leq \ \sum_{f=1}^{D} \  2^{\frac d6 (Nd^2+2N-3Nd+3d-3)} \cdot 2^{D \left (\frac{d(d-1)}{2}-1\right) } \cdot  4^{2D} \; \cdot \;  \sum_{\substack{m_1, m_2, \dots, m_f \\ \sum m_i = 2D \\ m_i \text{ even, } m_i\geq2 }} 1 \\
        \label{eq:bound4}
        &= \ 2^{\frac d6 (Nd^2+2N-3Nd+3d-3)  + D \left(\frac{d(d-1)}{2}-1\right)  +4D}  \; \cdot \;  \sum_{f=1}^{D} \binom{D-1}{f-1}  \\
        \label{eq:bound5}
         &=\ 2^{N\frac{d-1}{12}(5d^2-7d+24) +d(d-1)+3} \ .
    \end{align}
    Some explanation: Inequality (\ref{eq:bound3}) follows from the relations (\ref{eq:DNrel}) and from
    \begin{equation*}
        \sum_{i=1}^{f} m_i = 2D \quad \text{and }\quad \sum_{i=2}^{f} L_i = \frac{1}{2} \sum_{i=2}^{f} m_i \left ( \frac{d(d-1)}{2} -1\right )  \leq  \frac{1}{2} \,  2D \, \left( \frac{d(d-1)}{2} -1\right ) \ ,    
    \end{equation*}
    which hold because at the end, we glue all the $2D$ boundary faces. The inequality (\ref{eq:bound4}) follows from the fact that the number of compositions of $2D$ with $f$ parts all even, is the same as the number of compositions of $D$ with $f$ parts, which is $\binom{D-1}{f-1}$. Equality (\ref{eq:bound5}) is by Newton's binomial formula $\sum_{f=1}^{D}\binom{D-1}{f-1} = 2^{D-1}$, and the definition of $D$. 
       Consequently, via Inequality \ref{eq:numberOfTrees}, the number of 2-LC $d$-manifolds without boundary with $N$ facets is at most
    \begin{equation*}
       (d \, e)^N \cdot 2^{N \frac{d-1}{12}(5d^2-7d+24)} \cdot {2^{d(d-1)+3}}\: \:  < \: \: 2^{N\frac{5d^3}{12}} \cdot {2^{d(d-1)+3}}\: \:  < \: \: 2^{N\frac{d^3}{2}} .
    \end{equation*}
    The same proof works also for $d$-manifolds with boundary: we simply stop the matching process earlier. \qedhere
\end{proof}

\begin{remark} \label{rem:pseudo} The exponential bound of Theorem \ref{thm:main} cannot be extended from manifolds to pseudomanifolds. In fact, already for  $d = 3$, the family of the cones $v\ast S$, where $S$ is any trangulated surface  and $v$ is a new vertex, shows that 2-LC $3$-pseudomanifolds are more than exponentially many. Nevertheless, it is possible to expand Theorem \ref{thm:main} by only allowing those 2-LC pseudomanifolds that are obtained with gluings that satisfy all the conditions from Lemmas \ref{lemma:gluings} and \ref{lemma:connected_components}.   Formally: 
\end{remark}

\begin{definition} A \emph{$2$-LC quasimanifold} is any pseudomanifold obtainable from a tree of $d$-simplices by performing only 2-LC gluings that are 
    \begin{compactitem}
            \item orientable with respect to the involved links of $(d-3)$-faces,
            \item planar with respect to the involved links of $(d-3)$-faces,
            \item with an impact only on the boundaries of the links of $(d-3)$-faces, and
            \item not forcing any identification between two distinct connected components of the boundary of the link of a $(d-3)$-face.
    \end{compactitem}
    Moreover, we require any two $(d-1)$-faces that share an interior edge of the link of a $(d-3)$-face to be glued together. With the same proof of Theorem \ref{thm:main}, we conclude the following:
\end{definition}

\begin{theorem} \label{thm:mainQuasi}
    For any $d \geq 3$, the number of combinatorially distinct $2$-LC $d$-quasimanifolds 
    with $N$ facets, for $N$ large, is smaller than
\[ 
2^{\frac{d^3}{2}N} .
\]    
\end{theorem}

\begin{example} \label{ex:nonLC}
In the boundary of a tree of tetrahedra, choose two triangles $\sigma_1$ and $\sigma_2$ that intersect at a vertex and that  do not belong to the same tetrahedron or to adjacent tetrahedra. Let $P$ be the $3$-dimensional pseudomanifold obtained by gluing $\sigma_1$ and $\sigma_2$. Then $P$ is a $2$-LC quasimanifold. However, $P$ cannot be LC, because it is not homeomorphic to any of the possible topologies of LC $3$-pseudomanifolds, as characterized by Durhuus and Jonsson \cite[Theorem 2]{Durhuus}. This example highlights how the class of $2$-LC $d$-quasimanifolds, bounded by Theorem \ref{thm:mainQuasi}, is much larger than the class of LC $d$-quasimanifolds, already for $d=3$.
\end{example}


\begin{remark} \label{ex:nonLC} 
We conjecture that $3$-LC $d$-manifolds are more than exponentially many for every $d \ge 3$. Note that a much stronger statement, ``$3$-LC $d$-spheres are more than exponentially many for every $d \ge 3$'', would be immediately implied, via suspensions, by a positive solution to Gromov's problem of whether there are more than exponentially many $3$-spheres.
\end{remark}

\end{document}